\documentclass{amsart}
\usepackage{amsfonts,amssymb,amscd}
\usepackage[all]{xy}

\title{Spaces of maps into topological group with the Whitney topology}

\author[T. Banakh]{Taras Banakh}
 \address[T. Banakh]{Instytut Matematyki, Uniwersytet Humanistyczno-Przyrodniczy Jana
Kochanowskiego w Kielcach, Poland, and\newline
Department of Mathematics, 
Ivan Franko National University of Lviv, Lviv, 79000, Ukraine}
 \email{tbanakh@yahoo.com}

\author[K. Mine]{Kotaro Mine}
 \address[K. Mine]{Institute of Mathematics,
 University of Tsukuba, Tsukuba, 305-8571, Japan}
 \email{pen@math.tsukuba.ac.jp}

\author[K. Sakai]{Katsuro Sakai}
 \address[K. Sakai]{Institute of Mathematics,
 University of Tsukuba, Tsukuba, 305-8571, Japan}
 \email{sakaiktr@sakura.cc.tsukuba.ac.jp}

\author[T. Yagasaki]{Tatsuhiko Yagasaki}
\address[T. Yagasaki]{Division of Mathematics,
 Kyoto Institute of Technology, Kyoto, 606-8585, Japan}
\email{yagasaki@kit.ac.jp}

\subjclass[2000]
{46A13, 46T10, 54H11, 57N20, 58D15}
\keywords{Fr{\'e}chet space, LF-space, topological group, 
$l_2$-manifold, ($\IR^\infty\times l_2$)-manifold, function space, the Whitney (graph) topology, the box product, the small box product, end-discrete.} 

\thanks{This work is supported by Grant-in-Aid for Scientific Research (No.19540078).} 

\newcommand{\IR}{\mathbb R}
\newcommand{\IN}{\mathbb N}

\newcommand{\U}{\mathcal U}
\newcommand{\V}{\mathcal V}
\newcommand{\W}{\mathcal W}

\newcommand{\N}{\mathcal N}
\newcommand{\K}{\mathcal K}

\newcommand{\G}{\mathcal G}
\newcommand{\LL}{\mathcal L}

\newcommand{\supp}{\operatorname{supp}}
\newcommand{\cbox}{\boxdot}

\newcommand{\id}{\mathrm{id}}

\newcommand{\cl}{\mathrm{cl}}

\newcommand{\tint}{\mathrm{int}}

\newtheorem{theorem}{Theorem}[section]
\newtheorem{proposition}[theorem]{Proposition}
\newtheorem{lemma}[theorem]{Lemma}
\newtheorem{corollary}[theorem]{Corollary}

\theoremstyle{definition}

\newtheorem{remark}[theorem]{Remark}

\begin{document}

\maketitle

\begin{abstract}
Let $X$ be a locally compact Polish space and
 $G$ a non-discrete Polish ANR group.
By $C(X,G)$, we denote the topological group of
 all continuous maps $f:X\to G$ endowed with the Whitney (graph) topology and
 by $C_c(X,G)$ the subgroup consisting of all maps with compact support. 
It is known that if $X$ is compact and non-discrete
 then the space $C(X,G)$ is an $l_2$-manifold. 
In this article we show that 
 if $X$ is non-compact and not end-discrete
 then $C_c(X,G)$ is an $(\IR^\infty \times l_2)$-manifold,
 and moreover the pair $(C(X,G), C_c(X,G))$ is locally homeomorphic
 to the pair of the box and the small box powers of $l_2$. 
\end{abstract}

\section{Introduction}

This paper is one of studies on local or global topological types of infinite-dimensional topological groups, which appear as function spaces over non-compact spaces (\cite{Ban, BMS, BMSY, BMSY2, BY}). 
For spaces $X$ and $Y$ let $C(X,Y)$ denote the space of all continuous maps $f:X\to Y$ endowed with 
the {\em Whitney} ({\em graph\/}) {\em topology}. 
When $X$ is compact, this topology coincides with the compact-open topology. 
In \cite{Sa1} the third author showed that 
if $X$ is a non-discrete compact metrizable space and 
$Y$ is a Polish ANR without isolated points, then the space $C(X,Y)$ is an $l_2$-manifold, i.e., a paracompact Hausdorff space which is locally homeomorphic to ($\approx_\ell$) the separable Hilbert space $l_2$. 
Here and below, a Polish space means a separable completely metrizable space. 

Suppose $X$ is a paracompact space and 
$G$ is a Hausdorff topological group with the unit element $e$. 
Then, the space $C(X,G)$ is a topological group under the pointwise multiplication. 
Let $C_c(X,G)$ denote the subgroup of $C(X,G)$ consisting of maps $f : X \to G$ with compact support. 
Here, the support of $f \in C(X,G)$ is defined by $\supp(f) = \cl_X \{ x \in X : f(x) \neq e\}$.  
When $X$ is a non-discrete compact metrizable space and 
$G$ is a non-discrete Polish ANR group, 
 from \cite{Sa1} or the famous  
Dobrowolski\,-\,Toru{\'n}czyk's theorem on $l_2$-manifold topological groups, 
it follows that $C(X,G)$ is an $l_2$-manifold.
In this paper we study topological types of the spaces $C(X,G)$ and $C_c(X,G)$ in the case where $X$ is non-compact. The group structure on these spaces is a key ingredient to our arguments. 

A {\em Fr\'echet space} is a completely metrizable
 locally convex topological linear space.
An {\em LF-space} is the direct limit of increasing sequence of
 Fr\'echet spaces in the category of
 (locally convex) topological linear spaces.
In \cite{Man} it is shown that 
 a separable LF-space is homeomorphic to ($\approx$)
 either $l_2$, $\IR^\infty$ or $\IR^\infty \times l_2$, where $\IR^\infty$ is 
 the direct limit of the tower 
$$\IR^1\subset \IR^2\subset \IR^3\subset\cdots.$$
The spaces $\IR^\infty$ and $\IR^\infty \times l_2$ are homeomorphic to
 the countable small box powers $\cbox^\IN \IR$ and $\cbox^\IN l_2$,
 respectively.
The latter spaces are subspaces of
 the box powers $\square^\IN \IR$ and $\square^\IN l_2$.
See Section 2 for the definition of the (small) box powers (or products).
For simplicity, the pair $(\square^\IN l_2, \cbox^\IN l_2)$ is
 denoted by $(\square, \cbox)^\IN l_2$.
We say that a space $X$ is {\em end-discrete}
 if $X\setminus K$ is discrete for some compact subset $K$ of $X$. 
The following are the main results of this paper. 

\begin{theorem}\label{thm_loc-homeo}
Let $X$ be a non-compact locally compact Polish space 
 and $G$ a non-discrete Polish AR (ANR) group.
\begin{enumerate}
\item If $X$ is not end-discrete or $G$ is not locally compact, then 
$$(C(X,G), C_c(X,G)) \approx_{(\ell)} (\square, \cbox)^\IN l_2.$$
\item Suppose $G$ is locally compact. 
\begin{enumerate}
\item If $X$ is end-discrete and non-discrete, then 
$$(C(X,G), C_c(X,G)) \approx_{(\ell)} l_2 \times (\square, \cbox)^\IN \IR.$$
\item If $X$ is discrete (i.e., $X\approx \IN$), then 
$$(C(X,G), C_c(X,G)) \approx_{(\ell)} (\square, \cbox)^\IN \IR.$$
\end{enumerate}
\end{enumerate}
\end{theorem}

\begin{theorem}\label{thm_loc-homeo_2}
Let $X$ be a non-compact locally compact Polish space 
 and $G$ a Hausdorff topological group. 
\begin{itemize}
\item[(1)] If $G$ is (locally) contractible,
 then $C_c(X,G)$ is (locally) contractible. 
\item[(2)] Suppose $G$ is a non-discrete Polish ANR group. 
\begin{itemize}
\item[(i)\,] If $X$ is non-discrete or $G$ is not locally compact, then 
the space $C_c(X,G)$ is an $(\IR^\infty \times l_2)$-manifold.
In addition, if $G$ is an AR, then $C_c(X,G) \approx \IR^\infty \times l_2$. 
\item[(ii)] If $X$ is discrete and $G$ is locally compact, then 
the space $C_c(X,G)$ is an $\IR^\infty$-manifold.
In addition, if $G$ is an AR, then $C_c(X,G) \approx \IR^\infty$. 
\end{itemize}
\end{itemize}
\end{theorem}

%%%%%%%%%%%%%%

\section{The box topology and topological groups}

In this preliminary section we recall basic facts on the box topology and topological groups (cf.\ \cite[Section 2]{BMSY}, \cite{GZ, Wil1}). 
The {\em box product} \ $\square_{i\in\IN} X_i$
	of a sequence of topological spaces $(X_i)_{i\in\IN}$
	is the countable product $\prod_{i\in\IN}X_i$
	endowed with the box topology. 
	This topology is generated
	by the base consisting of boxes $\prod_{i\in\IN}U_i$,
	where $U_i$ is an open subset in $X_i$.
The {\em small box product} \ $\cbox_{i\in\IN} X_i$ of
 a sequence of pointed spaces $(X_i,*_i)_{i\in\IN}$
 is the subspace of $\square_{i\in\IN} X_i$ defined by
$$\cbox_{i\in\IN} X_i = \big\{\, (x_i)_{i\in\IN}\in\square_{i\in\IN} X_i :
	 \exists \,n\in\IN, \;\forall i\ge n,\  x_i=*_i \,\big\}.$$
For notational simplicity,  the pair $(\square_{i\in\IN} X_i, \cbox_{i\in\IN} X_i)$ is denoted by the symbol $(\square, \cbox)_{i\in\IN} X_i$. 
In case $X_i = X$ for every $i \in \IN$,
 $\square_{i\in\IN} X_i$ and $\cbox_{i\in\IN} X_i$
 are respectively denoted by $\square^\IN X$ and $\cbox^\IN X$
 and called the {\em box power} and the {\em small box power},
 and moreover $(\square, \cbox)_{i\in\IN} X_i$ is denoted by $(\square, \cbox)^\IN X$.

Suppose $G$ is a topological group with the unit element $e \in G$. 
We always choose this unit element $e$ as the distinguished point of $G$ and its subgroups. 
A tower of closed subgroups in $G$ is a sequence $(G_i)_{i\in\IN}$ of closed subgroups of $G$ such that
\[ G_1\subset G_2\subset G_3 \subset \cdots \ \ \text{ and } \ \ 
 G = \bigcup_{i \in \IN} G_i. \]
This tower yields a box product pair $(\square, \cbox)_{i\in\IN} G_i$ and 
the multiplication map 
\[ p :\cbox_{i\in\IN} G_i\to G, \ \
	 p(x_1, x_2, \dots,x_i) = x_1x_2 \cdots x_i. \]
Note that $(\square, \cbox)_{i\in\IN} G_i$ is a pair of a topological group and its subgroup with the unit element $\boldsymbol{e} = (e, e, \dots)$ and that  
the multiplication map $p$ is continuous (\cite[Lemma 2.10]{BMSY}). 

We say that (i) a map $f : X \to Y$ is {\em open at a point} $x \in X$ if
for any neighborhood $U$ of $x$ in $X$
 the image $f(U)$ is a neighborhood of $f(x)$ in $Y$ and 
 (ii) $f : X \to Y$ has a local section at $y \in Y$ if there exist a neighborhood $V$ of $y$ in $Y$ and a map $s : V \to X$ such that $fs = \id_V$. 

\begin{remark}\label{rem_open}
(1) The map $p$ is open iff $p$ is open at $\boldsymbol{e}$. 

(2) If the map $p$ has a local section at $e$, then $p$ is open. 
\end{remark}

We say that a topological group $G$ is {\em the direct limit}
 of the tower $(G_i)_{i\in\IN}$ in the category of topological groups
 if  any group homomorphism $h:G\to H$ to an arbitrary topological group $H$
 is continuous whenever its restrictions $h|_{G_i}$ are continuous for all $i\in\IN$.
Note that $G$ is the direct limit of $(G_i)_{i\in\IN}$ if and only if  
$G$ carries the strongest group topology
 inducing the original topology on each group $G_i$. 
This description of the direct limit topology has the following consequence (cf. \cite[Proposition 2.7]{BMSY2}, \cite[Theorem 2.2]{BR}):

\begin{proposition}\label{open then direct limit}
If the map $p :\cbox_{i\in\IN} G_i\to G$ is open,
 then $G$ is the direct limit of $(G_i)_{i\in\IN}$
 in the category of topological groups.
\end{proposition}

We conclude this section with some remarks on the (local) contractibility of the small box products. 
A space $X$ is called (i) {\em strongly locally contractible} at $x \in X$
 if every neighborhood $U$ of $x$ contains
 a neighborhood $V$ of $x$ which contracts to the point $x$ in $U$ with keeping $x$ fixed 
(i.e.,  there is a contraction $h_t : V \to U$ $(t \in [0,1]$) of $V$ with $h_t(x) = x$), and 
(ii) {\em strongly contractible} at $x \in X$ if 
$X$ contracts to $x$ in itself with keeping $x$ fixed. 
Note that if a topological group $G$
 is (locally) contractible 
 then $G$ is strongly (locally) contractible at every $x \in G$ (\cite[Remark 2.9]{BMSY}). 
 
\begin{proposition}[{\rm\cite[Proposition 2.8]{BMSY}}]\label{p_LC} 
The small box product\/ $\cbox_{i \in \IN} (X_i, x_i)$ is strongly (locally) contractible at the base point $\boldsymbol{x} = (x_i)_{i\in\IN}$ if and only if each space $X_i$ is strongly (locally) contractible at the base point $x_i$. 
\end{proposition}

%%%%%%%%%%%%%%%%%%%%%

\section{Topological group $C(X,G)$} 

\subsection{The graph topology}

For spaces $X$ and $Y$ let $C(X, Y)$ denote the space of continuous maps $f : X \to Y$ endowed with the {\em Whitney topology} (or the {\em graph topology}). 
For $f \in C(X,Y)$ the symbol $\Gamma_f$ stands for the graph of $f$ (i.e.,  
$\Gamma_f = \{\, (x,f(x)) \in X \times Y : x\in X\,\})$.
The graph topology on $C(X,Y)$ is generated by the base 
	$$\langle \U \rangle = \{\, f\in C(X,Y): \Gamma_f \subset \U\,\} \hspace{5mm} 
 	\text{($\U$ runs over the open sets in $X\times Y$).}$$ 
For pairs of spaces $(X, A)$ and $(Y, B)$ let 
$C(X,A; Y, B)$ denote the subspace of $C(X, Y)$ defined by 
$$C(X,A; Y, B) = \{\, f \in C(X, Y) : f(A) \subset B \,\}.$$ 
For a point $y \in Y$ let $\widehat{y}=\widehat{y}_X \in C(X, Y)$ denote the constant map onto $y$. 

%%%%%%%%%%%%%%%%%%%%%

Suppose $G$ is a topological group with the unit element $e$. 
Then, the space $C(X,G)$ has the canonical group structure defined by the pointwise multiplication 
$$(fg)(x)=f(x)g(x) \ \ (f,g\in C(X,G), \ x \in X).$$
Then, the constant map $\widehat e$ is the unit element of $C(X,G)$.
The inverse $f^{-1}$ of $f \in C(X, G)$ is defined by 
\[ f^{-1}(x) = \big(f(x)\big)^{-1} \ \ (x \in X). \] 
The support of $f \in C(X,G)$ is defined by 
$\supp(f) = \cl_X \{ x \in X : f(x) \neq e\}$.  We obtain a subgroup 
$$C_c(X,G) = \{ f \in C(X,G) : \supp(f) \ \text{is compact}\}.$$ 

Now we shall show that when $X$ is paracompact, $\G = C(X,G)$ is a topological group with respect to the graph topology.   
First we see that the inverse operator and 
the left/right multiplications are always continuous. 

\begin{lemma} \label{lem_inv-lrm} 
The inverse operator $\theta : \G \to \G$, $\theta(g) = g^{-1}$ and 
the left/right multiplications $L_f, R_f : \G \to \G$, $L_f(g) = fg$, $R_f(g) = gf$
are continuous. 
\end{lemma} 

\begin{proof} 
Since $G$ is a topological group, we have the homeomorphisms: 
\begin{itemize}
\item[(i)\,] $\Theta : X \times G \approx X \times G$, \ \ $\Theta(x,y) = (x, y^{-1})$
\item[(ii)] $\Phi_f, \Phi^f : X \times G \approx X \times G$, \ \ $\Phi_f(x,y) = (x, f(x)y)$, \ \ $\Phi^f(x,y) = (x, yf(x))$. 
\end{itemize} 
Note that $\Theta(\Gamma_g) = \Gamma_{g^{-1}}$, $\Phi_f(\Gamma_g) = \Gamma_{fg}$ and $\Phi^f(\Gamma_g)=\Gamma_{gf}$ for any $g\in C(X,G)$.
Therefore, if $g \in \langle \U \rangle$, then $g^{-1} \in \langle \Theta(\U) \rangle$ and $fg \in \langle \Phi_f(\U) \rangle$, 
$gf \in \langle \Phi^f(\U) \rangle$. This implies the assertions. 
\end{proof} 

\begin{proposition} \label{to be top group} 
For every paracompact space $X$, the group $C(X,G)$ is a topological group with respect to the graph topology. 
\end{proposition}

By Lemma~\ref{lem_inv-lrm} the continuity of the multiplication of $C(X,G)$ reduces to its continuity at $(\widehat e, \widehat e)$. The latter is equivalent to the assertion that each neighborhood $\langle \U \rangle$ of $\widehat e$ admits 
a smaller neighborhood $\langle \V \rangle$ of $\widehat e$ with 
$\langle \V \rangle \langle \V \rangle \subset \langle \U \rangle$. 
For $\U, \V \subset X\times G$, we set \ 
$\U\V =\{\, (x,yz)\in X\times G : (x,y)\in\U, (x,z) \in \V \, \}.$ \ 
Note that $\Gamma_{fg} = \Gamma_f \Gamma_g $ and $\langle \U \rangle \langle \V \rangle \subset \langle \U\V\rangle$.  
Therefore, Proposition~\ref{to be top group} follows from the next lemma.  
	
\begin{lemma} \label{multiplicative operator at e}
If $X$ is paracompact, then for any neighborhood $\U$ of $\Gamma_{\widehat e}$  
there exists a neighborhood $\V$ of $\Gamma_{\widehat e}$ such that $\V\V \subset \U$.
\end{lemma}

\begin{proof} 
For each $x \in X$ there exist open neighborhoods $U_x$ of $x$ in $X$ and $W_x$ of $e$ in $G$ such that $U_x \times W_x W_x \subset \U$. 
Since $X$ is paracompact, the open covering $\{ U_x \}_{x \in X}$ admits a locally finite open refinement $\{ V_x \}_{x \in X}$ with $V_x \subset U_x$ $(x \in X)$. 
Each $y \in X$ has an open neighborhood $O_y$ which meets at most finitely many $V_x$'s. Then $\widetilde W_y = \bigcap \{ W_x : O_y \cap V_x \neq \emptyset\}$ is an open neighborhood of $e$ in $G$. Finally, we define 
$$\V = \bigcup_{y \in X} O_y \times \widetilde W_y.$$

It remains to show that $\V\V \subset \U$. Given any $(z, a), (z, b) \in \V$. 
There exist $y, y', x \in X$ such that 
$(z, a) \in O_y \times \widetilde W_y$, $(z, b) \in O_{y'} \times \widetilde W_{y'}$ and $z \in V_x$. 
Since $V_x$ meets both $O_y$ and $O_{y'}$, 
it follows that $\widetilde W_y,  \widetilde W_{y'} \subset W_x$ and 
that $(z, ab) \in U_x \times W_x W_x \subset \U$. 
This means that $\V\V \subset \U$. 
\end{proof}

%%%%%%%%%%%%%%%%%%%%%%

\subsection{Compact case} 

When $X$ is compact, the graph topology on $C(X, Y)$ coincides with the compact-open topology. 
In this subsection we list some basic facts on the space $C(X, Y)$ for a compact metrizable space $X$.
The proof of the main theorem in \cite{Sa1} can be modified to prove the following relative version:

\begin{theorem}[{\rm \cite{Sa1}}]\label{thm_l_2}  
Let $X$ be a compact metrizable space with a compact subset $K \subsetneqq X$
 and $Y$ a Polish AR (ANR) without isolated points. 
If $X \setminus K$ is infinite or $Y \approx_\ell l_2$, then $C(X, K; Y, y) \approx_{(\ell)} l_2$ for any $y \in Y$. 
\end{theorem} 

When $Y$ is a topological group, we have a more precise conclusion.

\begin{theorem}[{\rm \cite{Gle, MZ}, \cite[Corollary 1]{DT}}]\label{thm_Lie} 
Let $G$ be a Polish AR (ANR) group. 
Then,
\begin{enumerate}
\item $G \approx_{(\ell)} l_2$ if $G$ is non-locally compact, and
\item $G \approx_{(\ell)} \IR^n$ for some $n \geq 0$
 if $G$ is locally compact.
\end{enumerate}
\end{theorem}

In Theorem~\ref{thm_Lie}, if $G$ is an ANR, then $G$ is an $l_2$-manifold in the case (1) and a Lie group in the case (2) (\cite[Corollary 1]{DT}). 
For the AR case, note that 
(1) any contractible $l_2$-manifold is homeomorphic to $l_2$ itself and 
(2) any contractible $n$-dimensional Lie group is homeomorphic to $\IR^n$. 
In fact, by Cartan-Malcev-Iwasawa's polar decomposition theorem \cite{Iw} 
any connected Lie group $G$ admits a factorization 
$G \approx K \times \IR^n$, where $K$ is any maximal compact subgroup of $G$. 
If $G$ is contractible, then the closed manifold $K$ is also contractible and hence consists of a single point. 

\begin{corollary} \label{C(X,G) is a mfd} 
Let $X$ be a compact metrizable space with a compact subset $K \subsetneqq X$ and
 $G$ a non-discrete Polish AR (ANR) group. 
\begin{enumerate}
\item If $X\setminus K$ is infinite or $G$ is non-locally compact, then $C(X,K;G,e) \approx_{(\ell)}l_2$. 
\item If $X\setminus K$ is finite and $G$ is locally compact, then $C(X,K;G,e) \approx_{(\ell)}\IR^n$ for some $n\ge1$.
\end{enumerate}
\end{corollary}

We also need some extension/deformation lemmas for maps.

\begin{lemma}\label{lem_ANR-ext}
Let $X$ be a compact metrizable space
 with $K, L \subset X$ disjoint compact subsets. 
\begin{enumerate}
\item If $Y$ is an ANR, then for any map $f : X \to Y$ 
there exist a neighborhood $\U$ of $f|_K$ in $C(K, Y)$ and a map $s : \U \to C(X, Y)$ such that $s(g)|_K = g$, $s(g)|_L = f|_L$ $(g \in \U)$ and $s(f|_K) = f$. 
\item If $Y$ is an AR, then we can take $\U=C(K, Y)$ in {\rm (1)}. 
\end{enumerate}
\end{lemma} 

\begin{proof} 
(1) We define a closed subset $H$ of the space $C(K,Y) \times X$
 and a map $\varphi : H \to Y$ by 
\begin{gather*}
H = \big(C(K,Y) \times (K \cup L)\big) \cup (\{f|_K\} \times X) \;\text{ and}\\
\varphi(g,x) = \begin{cases}
 g(x) &\text{if $x \in K$,}\\[0.5mm] 
 f(x) &\text{if $x \in L$ or $g = f|_K$.}
\end{cases}
\end{gather*}
Since $Y$ is an ANR and $C(K,Y) \times X$ is metrizable, 
the map $\varphi$ extends to a map $\tilde{\varphi} : \V \to Y$  over a neighborhood $\V$ of $H$ in $C(K,Y) \times X$.
Using compactness of $X$,
 	we can find a neighborhood $\U$ of $f|_K \in C(K,Y)$ such that $\U \times X  \subset \V$.
Then, the desired map $s : \U \to C(X,Y)$ is defined by 
$s(g)(x) = \tilde{\varphi}(g,x)$. 

(2) If $Y$ is an AR, then we can take $\V = C(K,Y) \times X$ and 
$\U = C(K,Y)$. 
\end{proof} 

\begin{lemma}\label{lem_LC-deform}
Suppose $X$ is a compact metrizable space and $K, L$ are 
disjoint compact subsets of $X$. 
\begin{enumerate}
\item 
If $Y$ is strongly locally contractible at a point $y \in Y$, then the following hold: 
\begin{enumerate}
\item
For any neighborhood $\V$ of $\widehat{y}_X$ in $C(X, Y)$  
there exist a neighborhood $\U$ of $\widehat{y}_X$ in $C(X, Y)$ 
and a homotopy $s_t : \U \to \V$ $(t \in [0,1])$ such that \ \ for each $f \in \U$ 
\begin{itemize}
\item[(a)] $s_0(f) = f$, \ \ {\rm (b)} $s_1(f) \in C(X, L; Y,y)$, 
\item[(c)] $s_t(f)|_{K} = f|_{K}$ $(t \in [0,1])$, 
\item[(d)] if $x \in X$ and $f(x) = y$, then $s_t(f)(x) = y$ $(t \in [0,1])$ \\ 
$($in particular, $s_t(\widehat y_X)=\widehat y_X${}$)$.  
\end{itemize}
\item
 For any subset $F \subset X$, the subspace $C(X, F; Y, y)$ is strongly locally contractible at $\widehat{y}_X$. 
\end{enumerate}
\item
If $Y$ is strongly contractible at $y$, then the following hold: 
\begin{enumerate}
\item  There exists a homotopy $s_t : C(X, Y) \to C(X, Y)$ $(t \in [0,1])$ 
which satisfies the consitions {\rm (1)(i)(a)\,-\,(d)} for each $f \in  C(X, Y)$. 
\item For any subset $F \subset X$, the subspace $C(X, F; Y, y)$ is strongly contractible at $\widehat{y}_X$. 
\end{enumerate}
\end{enumerate}
\end{lemma} 

\begin{proof} 
(1)\,(i) Take a map $\lambda : X \to [0,1]$ with $\lambda(K)=0$ and $\lambda(L)=1$. There exists a neighborhood $V$ of $y$ in $Y$ with  
$C(X, V) \subset \V$. 
By the assumption, there is a neighborhood $U$ of $y$ in $Y$ and a homotopy  
$h_t : U \to V$ $(t \in [0,1])$ such that $h_0 = \id_U$, $h_1= \widehat{y}_U$ and $h_t(y) = y$. Then $\U =C(X,U)$ is a neighborhood of $\widehat y_X$ and 
the required homotopy $s_t:\U \to \V$ is 
defined by $s_t(f)(x)= h_{\lambda(x)t}(f(x))$. 

(ii) We can apply (1)(i) to $K = \emptyset$ and $L = X$ so to obtain strong local contractions $s_t : \U \to \V$ at $\widehat y_X$ in $C(X,Y)$. The condition (d) implies that $s_t(\U \cap C(X, F; Y, y)) \subset \V \cap C(X, F; Y, y)$. 

(2)\,(i) The homotopy $s_t$ is defined in the same way as in (1)(i) 
by using a strong contraction $h_t$ of $Y$ at~$y$. 

(ii) The condition (d) assures that $s_t(C(X, F; Y, y)) \subset C(X, F; Y, y)$. 
\end{proof} 

%%%%%%%%%%%%%%%

\section{Topological type of $(C(X, G), C_c(X, G))$} 

In this section we prove Theorem~\ref{thm_loc-homeo}.  
To simplify the arguments, 
we use the following terminology: For pairs of spaces $(X, A)$ and $(Y, B)$, (a) we set 
$(X, A)\times (Y, B) = (X \times Y, A \times B)$, (b) we  
write 
$(X, A) \approx_\ell (Y, B)$ if for each point $a \in A$ there exist an open neighborhood $U$ of $a$ in $X$ and an open subset $V$ of $Y$ 
which admit a homeomorphism of pairs of spaces $(U, U \cap A) \approx (V, V \cap B)$, and (c) we say that a map $p : (X, A) \to (Y, B)$ has a local section at a point $b \in B$, 
if there exist an open neighborhood $V$ of $b$ in $Y$ and a map of pairs
$s : (V, V \cap B) \to (X, A)$ such that $ps = \id_V$.

Suppose $G$ is a Hausdorff topological group and 
$X$ is a locally compact Polish space. 
Let $\G = C(X, G)$ and $\G_c = C_c(X, G)$. 
For subsets $K, N \subset X$, let 
$$\begin{array}[c]{l}
\G_K = C(X, K; G,e), \ \ \ 
\G(N) = \G_{X \setminus N}, \\[1mm] 
\G_K(N) = \G_K \cap \G(N) \ \ \ \text{and} \ \ \ \G_c(N) = \G(N) \cap \G_c. 
\end{array}$$ 
Every discrete family $\LL = (L_i)_{i\in\IN}$ of compact subsets of $X$ induces two maps 
$$\begin{array}[c]{l}
r_\LL : (\G, \G_c) \to (\square, \cbox)_{i\in\IN} C(L_i,G),  \ \ \ r_\LL(f) = (f|_{L_i})_{i\in\IN}, 
\ \ \ \text{and} \\[2mm] 
\lambda_\LL : (\square, \cbox)_{i\in\IN} \G(L_i) \to (\G(L),\G_c(L)), 
\ \ \ \lambda_\LL((f_i)_{i\in\IN})|_{L_i} = f_i|_{L_i},  
\end{array}$$
where $L = \bigcup_{i\in\IN} L_i$.
Note that the map $\lambda_\LL$ is a homeomorphism. 

Let $\LL = (L_i)_{i\in\IN}$, $\N = (N_i)_{i\in\IN}$ and $\K = (K_i)_{i\in\IN}$ be discrete families of compact subsets of $X$ 
such that $L_i \subset {\rm Int}\,N_i$ $(i \in \IN)$ and $X = L \cup K$, where 
$L = \bigcup_{i\in\IN} L_i$, $K = \bigcup_{i\in\IN} K_i$ and $N = \bigcup_{i\in\IN} N_i$. The families $\N$ and $\K$ induces the homeomorphisms 
$$\lambda_\N : (\square, \cbox)_{i\in\IN} \G(N_i) \to (\G(N),\G_c(N)) 
	 \ \ \ \text{and} \ \ \ 
\lambda_\K: (\square, \cbox)_{i\in\IN} \G(K_i) \to (\G(K), \G_c(K)),$$ 
and we obtain the map 
\begin{gather*}
\rho : (\square,\cbox)_{i\in\IN}\G(N_i) \times
 (\square,\cbox)_{i\in\IN}\G(K_i) \to (\G, \G_c), \\ 
\rho((g_i)_{i\in\IN},(h_i)_{i\in\IN})
 = \lambda_\N((g_i)_{i\in\IN}) \cdot \lambda_\K((h_i)_{i\in\IN}).
\end{gather*}

\begin{lemma}\label{lem_LSP_G}
\begin{enumerate}
\item If $G$ is an AR (ANR), then
$$(\G, \G_c) \approx_{(\ell)} 
 (\square, \cbox)_{i\in\IN}C(L_i, G) \times (\square, \cbox)_{i\in\IN}\G_L(K_i).$$
\item If $G$ is (locally) contractible,
 then the map $\rho$ has a (local) section at $\widehat{e}_X$. 
\end{enumerate}
\end{lemma}

\begin{proof} 
(1) For each $i \in \IN$, since $L_i$ and $F_i={\rm Fr}_X N_i$ are
	disjoint compact sets in $N_i$, 
by Lemma~\ref{lem_ANR-ext}~(1), 
	there exist an open neighborhood $\V_i$ of 
	$\widehat{e}_{L_i} \in C(L_i, G)$ 
	and a map $s_i : \V_i \to C(N_i, F_i; G, e) \approx \G(N_i)$ 
	such that 
	$s_i(f)|_{L_i} = f$ $(f \in \V_i)$ and 
	$s_i(\widehat{e}_{L_i}) = \widehat{e}_{X}$. 
The maps $s_i$ $(i\in\IN)$ determine a continuous map
	$$s : \square_{i\in\IN} \V_i \to \square_{i\in\IN} \G(N_i),
 	\quad s((f_i)_{i\in\IN}) = \big(s_i(f_i)\big)_{i\in\IN}.$$ 
The preimage $\V=r_\LL^{-1}(\square_{i\in\IN} \V_i)$ is 
	an open neighborhood of $\widehat e_X$ in $\G$.
Consider the map 
$$\eta = \lambda_\N \circ s \circ r_\LL:\V \to \G(N).$$ 
Since $\eta(g)|_{L_i} = g|_{L_i}$ $(i \in \IN)$, 
	we have $\eta(g)^{-1} g \in \G_L$.
Hence, we obtain two maps  
$$\begin{array}[c]{l}
\phi : \V \to \square_{i\in\IN} \V_i\times \G_L, \
\phi(g) = (r_\LL(g), \eta(g)^{-1} \cdot g) \; \text{ and}\\[2mm] 
\theta:\square_{i\in\IN} \G(N_i)\times \G_L \to \G, \
\theta((g_i)_{i\in\IN}, h) = \lambda_{\N}((g_i)_{i\in\IN}) \cdot h.
\end{array}$$ 
It follows that 
$$\theta \circ (s \times \id) \circ \phi = \id_\V.$$
Indeed, we have \ \ 
	$(\theta \circ (s \times \id) \circ \phi)(g) = 
	(\lambda_\N \circ s \circ r_\LL)(g) \cdot \eta(g)^{-1} g = g$ \ \ 
	$(g \in \V).$ \ \ 
Next, consider the map $\psi_0 = \theta \circ (s \times \id): \square_{i\in\IN} \V_i \times \G_L \to \G$.
For any $((f_i)_{i\in\IN},h)\in \square_{i\in\IN} \V_i \times \G_L$, we have
\begin{align*}
\psi_0((f_i)_{i\in\IN},h)|_{L_i}
	&=\theta((s_i(f_i))_{i\in\IN},h)|_{L_i}
	=\lambda_\N ((s_i(f_i))_{i\in\IN})|_{L_i} \cdot h|_{L_i} \\
	&=\lambda_\N ((s_i(f_i))_{i\in\IN})|_{L_i}=f_i \in \V_i.
\end{align*}
This means that $(r_\LL\circ \psi_0)((f_i)_{i\in\IN},h) = (f_i)_{i\in\IN}$ and 
$\psi_0((f_i)_{i\in\IN},h) \in r_\LL^{-1}(\square_{i\in\IN} \V_i)=\V$. 
Thus, we obtain the map 
$$\psi=\theta \circ (s \times \id) : \square_{i\in\IN} \V_i \times \G_L \to \V.$$

Now we shall show that the maps $\phi$ and $\psi$ are reciprocal homeomorphisms.
It remains to show that $\phi\circ \psi=\id$.
For any $((f_i)_{i\in\IN},h)\in \square_{i\in\IN} \V_i \times \G_L$, if we put $g = \psi((f_i)_{i\in\IN},h)$, then we have  (a) $r_\LL(g) = (f_i)_{i\in\IN}$,
(b) $\eta(g) = (\lambda_\N \circ s \circ r_\LL)(g) = (\lambda_\N \circ s)((f_i)_{i\in\IN})$ and (c) 
$g = \psi((f_i)_{i\in\IN},h) = \theta(s((f_i)_{i\in\IN}), h) = (\lambda_\N \circ s)((f_i)_{i\in\IN}) \cdot h$.  
Therefore, it follows that 
$(\phi \circ \psi)((f_i)_{i\in\IN},h) = \phi(g) = (r_\LL(g), \eta(g)^{-1}g) = ((f_i)_{i\in\IN},h)$. 

Since $s_i(\widehat e_{L_i}) = \widehat e_X$ ($i \in \IN$), it is seen that
	$\phi(\V \cap \G_c) = \cbox_{i\in\IN} \V_i \times \G_{L, c}$ \ and 
$$\phi : (\V, \V \cap \G_c) \approx(\square, \cbox)_{i\in\IN} \V_i \times (\G_L, \G_{L, c}).$$
Since $X = L \cup K$, we have $\G_L \subset \G(K)$.
Thus, $\lambda_\K$ restricts to the homeomorphism 
$$\lambda_{\K} : (\square, \cbox)_{i\in\IN} \G_L(K_i) \approx (\G_L, \G_{L, c}).$$ 
Thus, we have 
	$(\V, \V \cap \G_c) \approx (\square, \cbox)_{i\in\IN} \V_i  
		\times (\square, \cbox)_{i\in\IN} \G_L(K_i)$.

If $G$ is an AR, then we can take $\V_i=C(L_i,G)$ and $\V=\G$.
Then, we have
	$$(\G, \G_c) \approx (\square, \cbox)_{i\in\IN} C(L_i,G) 
	\times (\square, \cbox)_{i\in\IN} \G_L(K_i).$$

(2) The argument proceeds along the same line as the first half part of (1). 
The family $\N$ induces the map 
$$r_\N : \G \to \square_{i\in\IN} C(N_i,G), \ \ \ r_\N(f)= (f|_{N_i})_{i\in\IN}.$$
By Lemma~\ref{lem_LC-deform}~(1), for each $i \in \IN$,
	there exist an open neighborhood $\V'_i$ of 
	$\widehat{e}_{N_i} \in C(N_i, G)$
	and a map $s'_i : \V_i' \to C(N_i, F_i; G, e) \approx \G(N_i)$
	such that $s'_i(f)|_{L_i} = f|_{L_i}$ $(f \in \V'_i)$ and 
	$s'_i(\widehat{e}_{N_i}) = \widehat{e}_{N_i}$. 
The maps $(s'_i)_{i\in\IN}$ determine a continuous map
	$$s' : \square_{i\in\IN} \V'_i \to \square_{i\in\IN} \G(N_i),
 	\quad s'((f_i)_{i\in\IN}) = \big(s'_i(f_i)\big)_{i\in\IN}.$$ 
As before, $\V'=r_\N^{-1}(\square_{i\in\IN} \V'_i)$ is 
	an open neighborhood of $\widehat e_X$ in $\G$ and we have maps  
$$\begin{array}[c]{l}
\eta' = \lambda_\N \circ s' \circ r_\N : \V' \to  \G(N) \ \ \ \text{and} \\[2mm] 
\phi':\V'\to\square_{i\in\IN} \V'_i \times \G_L, \ \ \phi'(g)=(r_\N(g),\eta'(g)^{-1}g).
\end{array}$$ 
Note that the map $\rho$ has the factorization 
\vspace{-2mm} 
$$\begin{array}[t]{ccccc}
	& \id \times \lambda_\K & & \theta & \\
	\rho : \ \square_{i\in\IN} \G(N_i) \times \square_{i\in\IN} \G_L(K_i) & \approx 
	& \square_{i\in\IN} \G(N_i) \times \G_L & \longrightarrow & \G. 
	\end{array} \medskip$$
Again, it is seen that $\theta \circ (s' \times \id) \circ \phi' = \id_{\V'}$, and hence 
the map $\theta$ has a local section at $\widehat{e}_X$ : 
$$(s' \times \id) \phi' :(\V', \V' \cap \G_c) 
	\to (\square, \cbox)_{i\in\IN} \G(N_i) \times (\G_L, \G_{L, c}).$$
Thus, the map $\rho$ also has a local section at $\widehat{e}_X$. 

If $G$ is contractible, then we can take $\V_i'=C(N_i,G)$ and $\V'=\G$.
Hence, $\rho$ has a global section.
This completes the proof. 
\end{proof}

A subset $A$ of $M$ is {\em regular closed} if $A = \cl_M (\tint_M A)$. 
Since $X$ is locally compact and $\sigma$-compact, there exists 
	a sequence $(X_i)_{i\in\IN}$ of compact regular closed subsets of 
	$X$ such that
	$X_ i \subset \tint_X X_{i+1}$ ($i \in\IN$) and $X=\bigcup_{i\in\IN}X_i$.
It induces the tower $(\G(X_i))_{i\in\IN}$ of closed subgroups of $\G_c$
	 and the multiplication map
\[ p : \cbox_{i\in\IN} \G(X_i) \to \G_c,
 		\quad p(h_1, \dots, h_n) = h_1 \cdots h_n. \] 
Let $L_i = X_i \setminus \tint_X X_{i-1}$ for each $i \in \IN$,
	where $X_0=\emptyset$.
Since each $X_i$ is regular closed, $L_i$ is also regular closed. 
We call the sequence $(X_i, L_i)_{i\in\IN}$ an {\em exhausting sequence} for $X$.

\begin{proposition}\label{p_loc-homeo}
Suppose $(X_i, L_i)_{i\in\IN}$ is an exhausting sequence for $X$. 
\begin{enumerate}
\item If $G$ is an AR (ANR), then
$$(\G, \G_c) \approx_{(\ell)} \big(\square, \cbox)_{i\in\IN} C(L_{2i}, G) 
	\times (\square, \cbox)_{i\in\IN} \G(L_{2i-1}).$$
\item If $G$ is (locally) contractible, then 
\begin{enumerate}
\item the map $p : \cbox_{i\in\IN} \G(X_i) \to \G_c$
 has a (local) section $s$ 
 at $\widehat{e}_X$ such that $s(\widehat{e}_X) = (\widehat{e}_X,\widehat{e}_X,\dots)$, 
\item the group $\G_c$ is the direct limit 
	of the tower $(\G(X_i))_{i\in\IN}$
 	in the category of topological groups.
\end{enumerate}
\end{enumerate}
\end{proposition}

\begin{proof}
There exists a sequence $(N_i)_{i\in\IN}$ of compact subsets of $X$
	such that $L_i \subset \tint_X N_i$ and $N_i \cap N_j \neq \emptyset$
	iff $|i - j| \leq 1$.
We apply Lemma \ref{lem_LSP_G} to the discrete families 
 $\LL = (L_{2i})_{i\in\IN}$, $\N = (N_{2i})_{i\in\IN}$
 and $\K = (L_{2i-1})_{i\in\IN}$.
Let $L = \bigcup_{i\in\IN} L_{2i}$.
Since $L_{2i-2}$ is regular closed,
 we have $\G_{L}(L_{2i-1}) = \G(L_{2i-1})$ for each $i \in \IN$.
\smallskip

(1) The statement follows directly from Lemma \ref{lem_LSP_G}\,(1).
\smallskip

(2)(i) By Lemma \ref{lem_LSP_G}\,(2) the map
\begin{gather*}
\rho : (\square, \cbox)_{i\in\IN} \G(N_{2i}) 
	\times (\square, \cbox)_{i\in\IN} \G(L_{2i-1})
 	\longrightarrow (\G, \G_c), \\ 
 	\rho((f_i)_{i\in\IN}, (g_i)_{i\in\IN}) = 
 		\lambda_{\N}((f_i)_{i\in\IN}) \cdot \lambda_{\K}((g_i)_{i\in\IN})
\end{gather*}
	has a local section 
	$S : (\V, \V \cap \G_c) \to (\square, \cbox)_{i\in\IN} \G(N_{2i}) 
		\times (\square, \cbox)_{i\in\IN} \G(L_{2i-1})$ at $\widehat{e}_X$
		such that  
	$$S(\widehat{e}_X) = 
		((\widehat{e}_X,\widehat{e}_X,\dots), 
		(\widehat{e}_X,\widehat{e}_X,\dots)).$$
Note that $\W = \V \cap \G_c$ is an open neighborhood of 
	$\widehat{e}_X$ in $\G_c$ 
 	and that for each $h \in \W$
 	the image $S(h) = ((f_i)_{i\in\IN}, (g_i)_{i\in\IN})$
 	satisfies the following conditions:
\begin{itemize}
\item[(a)]
$h = \lambda_{\N}((f_i)_{i\in\IN}) \cdot \lambda_{\K}((g_i)_{i\in\IN})
 = (f_1f_2\cdots)(g_1g_2\cdots) = f_1g_1f_2g_2f_3g_3 \cdots$.
\item[(b)]
$f_i \in \G(N_{2i}) \subset \G(X_{2i+1})$,
$g_i \in \G(L_{2i-1}) \subset \G(X_{2i-1}) 
	\subset \G(X_{2i+2})$ for each $i\in\IN$.
\item[(c)]
 $\big(\widehat{e}_X, \widehat{e}_X, f_1, g_1, f_2, g_2, \dots\big)
 	\in \cbox_{i\in\IN} \G(X_i)$ and 
 	$h = p\big(\widehat{e}_X, \widehat{e}_X, f_1, g_1, f_2, g_2, \dots\big)$.
\end{itemize}
Therefore, the required local section of the map $p : \cbox_{i\in\IN} \G(X_i) \to \G_c$
 is defined by 
$$s : \W \to \cbox_{i\in\IN} \G(X_i),\quad
 s(h) = \big(\widehat{e}_X, \widehat{e}_X, f_1, g_1, f_2, g_2, \dots\big).$$
If $G$ is contractible, then we can take $\V = \G$ and $\W = \G_c$. 

(ii) From (i) and Remark~\ref{rem_open}, it follows that $p$ is open at 
	$\boldsymbol{e}=(\widehat e_X, \widehat e_X,\dots)$.
Thus, the assertion follows from Proposition \ref{open then direct limit}.
This completes the proof. 
\end{proof}

\begin{proof}[\bf Proof of Theorem~\ref{thm_loc-homeo}] 
By Proposition~\ref{p_loc-homeo}\,(1) it only remains to determine the topological type of the spaces $C(L_{2i}, G)$ and $\G(L_{2i-1}) \approx C(L_{2i-1},{\rm Fr}_X L_{2i-1}; G,e)$ for a suitable exhausting sequence $(X_i, L_i)_{i\in\IN}$ for $X$.

(1)
If $G$ is non-locally compact, we choose any exhausting sequence $(X_i, L_i)_{i\in\IN}$. If $X$ is not end-discrete, we can find an exhausting sequence $(X_i, L_i)_{i\in\IN}$ such that ${\rm Int}_XL_i$ is infinite for each $i \geq 1$. 
In each case we have $C(L_{2i}, G) \approx_{(\ell)} l_2$ and $\G(L_{2i-1}) \approx_{(\ell)}l_2$ $(i \geq 1)$ by Corollary~\ref{C(X,G) is a mfd}\,(1). 

(2) Since $G$ is locally contractible, it follows that $G$ is a Lie group of dimension $n \geq 1$ (so that $G \approx_\ell \IR^n$) and that $G \approx \IR^n$ if $G$ is contractible. 
(i) By the assumption, we can find a compact open subset $X_1$ of $X$ such that 
$X_1$ is infinite and $X - X_1$ is discrete. 
Thus we have an exhausting sequence $(X_i, L_i)_{i\in\IN}$ such that 
(a) $L_i$'s are disjoint open subsets of $X$,  
(b) $L_1$ is infinite and (c) $L_i$ is a one-point set for any $i\ge2$. 
Then, $\G(L_{1}) \approx_{(\ell)}l_2$ and 
$C(L_{2i}, G) \approx \G(L_{2i+1})  \approx G$ $(i\ge1)$. \\ 
(ii) Choose an exhausting sequence $(X_i, L_i)_{i\in\IN}$ such that each $L_i$ is a one-point set.
Then, we have $\G(L_{2i-1}) \approx C(L_{2i}, G) \approx G$ $(i \geq 1)$.
\end{proof}

\begin{proof}[\bf Proof of Theorem~\ref{thm_loc-homeo_2}] 
(1) Choose any exhausting sequence $(X_i, L_i)_{i\in\IN}$ for $X$. 
Since $G$ is (locally) contractible, so is each $\G(X_i) \approx 
C(X_i, {\rm Fr}_X X_i; G, e)$ by Lemma~\ref{lem_LC-deform}. 
Hence, $\cbox_{i\in\IN} \G(X_i)$ is (locally) contractible from Proposition~\ref{p_LC}. 
By Proposition~\ref{p_loc-homeo}\,(2) the map $p : \cbox_{i\in\IN} \G(X_i) \to \G_c$ admits a (local) section $s$ at $\widehat{e}_X$ 
such that $s(\widehat{e}_X) = (\widehat{e}_X,\widehat{e}_X,\dots)$. 
Thus, the group $\G_c$ is also seen to be (locally) contractible. 

(2) The paracompactness of the space $C_c(X,G)$ follows from  
Proposition~\ref{prop_paracompact} below. 
Thus, the assertions follow from Theorem~\ref{thm_loc-homeo}. 
\end{proof}

Let $V$ be a topological linear space and $(V_i)_{i\in\IN}$ a tower of closed linear subspace of $V$.
Then, they are abelian groups and the direct limit of the tower $(V_i)_{i\in\IN}$ in the category of topological groups is also a topological linear space.
Thus, we have the following as a corollary of Proposition~\ref{p_loc-homeo}\,(2)(ii).

\begin{corollary}\label{it is LF-space}
Suppose $X$ is a locally compact Polish space and 
$(X_i)_{i\in\IN}$ is a sequence  of compact subsets of $X$ such that 
$X_ i \subset \tint_X X_{i+1}$ $(i \in\IN)$ and $X=\bigcup_{i\in\IN}X_i$. 
Then, for any Fr\'echet space $F$ the space $C_c(X,F)$ 
is the LF-space with respect to the tower of Fr\'echet spaces $(C(X_i,{\rm Fr}_X X_i; F,0))_{i\in\IN}$. 
\end{corollary}

Let $(Y,*)$ be a pointed space.
Similarly to the above, 
 the support of $f \in C(X,Y)$ is defined by 
$\supp(f) = \cl_X \{ x \in X : f(x) \neq \ast\}$ and we define a space $C_c(X,Y)$ as follows:
$$C_c(X,Y) = \{ f \in C(X,G) : \supp(f) \ \text{is compact}\}.$$ 

A space is said to be {\em perfectly paracompact} if it is paracompact and each open subset is of type $F_\sigma$. 
It should be noticed that any subspace of a perfectly paracompact space is perfectly paracompact 
(cf.\ \cite[Theorem 5.1.28]{En}, \cite[Chapter VI\!I\!I, Section 2, 2.5]{Dug}). 
From \cite[Section 4]{Dou} it follows that the small box product of metrizable spaces is perfectly paracompact. 
Since any LF-space is homeomorphic to a small box product of Fr\'echet spaces \cite{Man}, 
it follows that every LF-space is perfectly paracompact. 
Hence, Corollary~\ref{it is LF-space} has the following consequence. 

\begin{proposition}\label{prop_paracompact} 
Suppose $X$ is a locally compact Polish space and $(Y,*)$ is a pointed metrizable space.
Then, $C_c(X,Y)$ is perfectly paracompact.
\end{proposition}

\begin{proof} 
Let $(N, 0)$ be a Banach space which contains $(Y,*)$ as a subspace. 
By Corollary~\ref{it is LF-space} the space $C_c(X,N)$ is an LF-space, 
and hence it is perfectly paracompact. 
Since $C_c(X,Y)$ is a subspace of $C_c(X,N)$, we have the conclusion. 
\end{proof}

\end{document}